\documentclass[11pt]{amsart}
\usepackage{amscd}
\usepackage{amsmath}
\usepackage{amssymb}
\usepackage{latexsym}
\usepackage{lscape}
\usepackage{xypic}

\newtheorem{thm}{Theorem}[section]
\newtheorem{prop}[thm]{Proposition}
\newtheorem{lem}[thm]{Lemma}
\newtheorem{lem-def}[thm]{Lemma-Definition}
\newtheorem{cor}[thm]{Corollary}
\theoremstyle{remark}

\newtheorem{rmk}{Remark}[section]
\theoremstyle{definition}

\input xy
\xyoption{all}

\newcommand{\nc}{\newcommand}
\nc{\on}{\operatorname}
\newcommand{\quash}[1]{} 

\newcommand{\frakb}{{\mathfrak b}}
\newcommand{\frakc}{{\mathfrak c}}

\newcommand{\frakg}{{\mathfrak g}}

\newcommand{\frakl}{{\mathfrak l}}

\newcommand{\frakn}{{\mathfrak n}}

\newcommand{\frakp}{{\mathfrak p}}

\newcommand{\frakt}{{\mathfrak t}}
\newcommand{\fraku}{{\mathfrak u}}

\newcommand{\bbA}{{\mathbb A}}
\newcommand{\bbB}{{\mathbb B}}
\newcommand{\bbC}{{\mathbb C}}
\newcommand{\bbD}{{\mathbb D}}
\newcommand{\bbF}{{\mathbb F}}
\newcommand{\bbG}{{\mathbb G}}

\newcommand{\bbQ}{{\mathbb Q}}

\newcommand{\bbZ}{{\mathbb Z}}

\newcommand{\calB}{{\mathcal B}}

\newcommand{\calF}{{\mathcal F}}

\newcommand{\calH}{{\mathcal H}}
\newcommand{\calI}{{\mathcal I}}

\newcommand{\calL}{{\mathcal L}}

\newcommand{\calO}{{\mathcal O}}
\newcommand{\calP}{{\mathcal P}}

\newcommand{\calS}{{\mathcal S}}

\newcommand{\calW}{{\mathcal W}}

\nc{\al}{{\alpha}} \nc{\be}{{\beta}}
\newcommand{\ga}{{\gamma}}
\nc{\Ga}{{\Gamma}}
\newcommand{\la}{{\lambda}}
\nc{\La}{{\Lambda}}

\nc{\ad}{{\on{ad}}}
\newcommand{\Ad}{{\on{Ad}}}
\nc{\Aff}{{\mathbf{Aff}}}

\nc{\der}{{\on{der}}}

\nc{\Fl}{{\calF\ell}}

\newcommand{\Gr}{{\on{Gr}}}
\newcommand{\Hom}{{\on{Hom}}}
\newcommand{\id}{{\on{id}}}
\nc{\Id}{{\on{id}}}
\newcommand{\ind}{{\on{ind}}}
\nc{\Ind}{{\on{Ind}}}
\newcommand{\Lie}{{\on{Lie}}}

\newcommand{\pr}{{\on{pr}}}

\newcommand{\Spec}{{\on{Spec}}}
\newcommand{\St}{{\on{St}}}
\nc{\tr}{{\on{tr}}}

\newcommand{\Mod}{{\mathrm{-Mod}}}

\newcommand{\bGr}{{\overline{\Gr}}}

\newcommand{\ppart}{(\!(t)\!)}

\numberwithin{equation}{section}

\title[Example of the derived Satake correspondence]{An example of the derived geometrical Satake correspondence over integers}
\address{Department of Mathematics, Harvard University, MA 02138}\email{xinwenz@math.harvard.edu}
\author{Xinwen Zhu}
\date{October, 2009}
\begin{document}
\maketitle
\begin{abstract}Let $G^\vee$ be a complex simple algebraic group. We describe certain morphisms of $G^\vee(\calO)$-equivariant complexes of sheaves on the affine Grassmannian $\Gr$ of $G^\vee$ in terms of certain morphisms of $G$-equivariant coherent sheaves on $\frakg$, where $G$ is the Langlands dual group of $G^\vee$ and $\frakg$ is its Lie algebra. This can be regarded as an example of the derived Satake correspondence.
\end{abstract}

\section*{Introduction}
Let $G^\vee$ be a complex reductive group, and $\Gr$ be its affine Grassmannian. Let $k$ be a commutative unital noetherian ring of finite global
dimension. The geometrical Satake correspondence (cf.
\cite{G,MV}) asserts that the Satake category $\calP_k$ of
$G^\vee(\calO)$-equivariant perverse sheaves with $k$-coefficients
on $\Gr$ is equivalent as a tensor category to the category of
representations of the Langlands dual group $G_k$ defined over $k$.
The Satake category arises as the heart of the natural t-structure on a triangulated category, the equivariant derived category $D_{G^\vee(\calO)}(\Gr)$. And it is a natural question, in particular asked by Drinfeld, to describe this so-called derived Satake category in terms of $G_k$\footnote{Indeed, the derived Satake category $D_{G^\vee(\calO)}(\Gr)$ has more structures than barely a triangulated category, and Drinfeld asked to describe these finer structures.}. When the coefficients
$k$ is a field of characteristic zero, a solution is given in
(cf. \cite{ABG, BeF}). Namely,
\begin{equation}\label{derived}
DS: D_{G^\vee(\calO)}(\Gr)\cong D_{perf}^G(S\frakg^*).\end{equation}
Here, $\frakg^*$ is the dual of the Lie algebra $\frakg$ of $G_k$, and $D_{perf}^G(S\frakg^*)$ is a "differential graded version" of the derived category of $G$-equivariant coherent sheaves on $\frakg$. The compatibility of the equivalence \eqref{derived} with the geometrical Satake isomorphism is as follows. Let $\calF$ be a $G^\vee(\calO)$-equivariant perverse sheaf on $\Gr$, and $V=H^*(\Gr,\calF)$ be the representation of $G_k$ under the geometrical correspondence. Then $DS(\calF)\cong V\otimes S\frakg^*$, where the grading(resp. $G$-action) on $V\otimes S\frakg^*$ is the product grading (resp. $G$-action).

Since the geometrical Satake correspondence holds for $k=\bbZ$, it
is natural to expect that the equivalence \eqref{derived} or at least the functor $DS$ should extend to integers,
after possibly inverting some small primes\footnote{The author was informed that Dennis Gaitsgory and Jacob Lurie have made significant progress to describe the full structure of the derived Satake category with arbitrary coefficients.}. Let us be a little more precise. Let $\bbZ_S$ denote the ring of integers after inverting some small primes. Then one can attach every Chevalley group scheme $G$ over $\bbZ_S$ its regular centralizer group scheme $J$ (see \cite{Ng} or \S \ref{regular} for its definition). According to a result of \cite{YZ}, the equivariant hypercohomology functor $H_{G^\vee(\calO)}$ gives a natural functor $D_{G^\vee(\calO)}(\Gr)\to D(J\Mod)$. On the other hand, there is a tautological functor $D^G(S\frakg^*)\to D(J\Mod)$ (see \eqref{fun}). One naturally expects that there is a functor $DS:D_{G^\vee(\calO)}(\Gr)\to D^G(S\frakg^*)$ that lifts $H_{G^\vee(\calO)}$.
Although such a lifting is not known to the
author at the moment, the goal of this note is to compare some natural morphisms in $D_{G^\vee(\calO)}(\Gr)$ with some natural morphisms in $D^G(S\frakg^*)$ by identifying their images in $D(J\Mod)$.

\section{Main results and notations}
\subsection{}We will assume that $G^\vee$ is simple and of adjoint type over $\bbC$. Let $\calO=\bbC[[t]]$ and $F=\bbC\ppart$. Then the affine
Grassmannian $\Gr=G^\vee(F)/G^\vee(\calO)$ of $G^\vee$ is a union of
projective varieties. To see this, we fix a Borel subgroup
$B^\vee\subset G^\vee$ and a maximal torus $T^\vee\subset B^\vee$.
Then each coweight $\lambda$ of $T^\vee$ determines a point
$t^\lambda\in T^\vee(F)$, and hence a point in $\Gr$, which we still
denote by $t^\lambda$. Let $\Gr^\lambda$ be the
$G^\vee(\calO)$-orbit through $t^\lambda$. Each
$G^\vee(\calO)$-orbit of $\Gr$ contains a unique point $t^\lambda$
for some {\em dominant} coweight $\lambda$. We denote the closure of
$\Gr^\lambda$ by $\bGr^\lambda$. Then each $\bGr^{\lambda}$ is a
projective variety and $\Gr$ is their union.

We denote $i^\la:\Gr^\la\to\Gr$ to be the natural locally closed
embedding. Then $i^{\la}_!:=i^{\la}_!\underline{\bbZ}[\dim\Gr^\la]$ and
$i^{\la}_*:=i^{\la}_*\underline{\bbZ}[\dim\Gr^\la]$ are naturally objects in
$D_{G^\vee(\calO)}(\Gr,\bbZ)$. In addition, their degree zero
perverse cohomology $I^\la_!:=\ ^pi^\la_!$ and $I^\la_*=\ ^pi^\la_*$
are the standard and the costandard objects in $\calP_\bbZ$. There
is a natural sequence of maps $i^\la_!\to I^\la_!\to I^\la_*\to i^\la_*$, which gives
\begin{equation}\label{seq}
H_{G^\vee(\calO)}(\Gr,i^\la_!)\to H_{G^\vee(\calO)}(\Gr,I^\la_!)\to H_{G^\vee(\calO)}(\Gr,I^\la_*)\to H_{G^\vee(\calO)}(\Gr,i^\la_*).
\end{equation}

Let us recall the description of the integral
homology of $\Gr$ in terms of $G_{\bbZ}$, the Langlands dual group of $G^\vee$ over $\bbZ$ (see below), following \cite{YZ} (the
description of the rational cohomology of $\Gr$ in terms of $G_\bbQ$
was obtained by Ginzburg (cf. \cite{G})). The
$G^\vee(\calO)$-equivariant homology of $\Gr$ is a commutative and
cocommutative Hopf algebra over $H(\bbB G^\vee,\bbZ)$, and therefore
$J=\Spec H_*^{G^\vee(\calO)}(\Gr,\bbZ)$ is a commutative group
scheme over $\Spec H(\bbB G^\vee,\bbZ)$.

Let $S$ be the multiplicative
set generated by the bad primes of $G^\vee$ (i.e., those dividing the coefficients of the highest root in terms of a linear combination of simple roots) and those dividing $n+1$
if $G^\vee$ is of type $A_n$. Let $\bbZ_S$ be the localization of $\bbZ$ by $S$. Then according to \cite{YZ}, $J$ is
canonically isomorphic to the regular centralizer of $G_{\bbZ_S}$ over
$\bbZ_S$ (see \S \ref{regular} for the review of the regular centralizer).
Observe that for any $\calF\in D_{G^\vee(\calO)}(\Gr,\bbZ_S)$, the
hypercohomology $H^*_{G^\vee(\calO)}(\Gr,\calF)$ is a module over
$J$. The goal of the note is to describe the sequence \eqref{seq} as $J$-modules.

\subsection{}
For brevity, we will suppress the subscript $\bbZ_S$ so that we write
$G$ for $G_{\bbZ_S}$, which is a simply-connected Chevalley group over $\bbZ_S$. Let $T\subset B$ be the maximal torus and the
Borel subgroup of $G$ dual to $T^\vee\subset B^\vee$. Let
$\frakt\subset\frakb\subset\frakg$ be their Lie algebras. Let
$\lambda$ be a dominant weight of $G$ w.r.t. $B$, and
$P_\lambda\supset B$ be the standard parabolic subgroup of $G$
corresponding to $\lambda$. That is, the Weyl group of $P_\la$
coincides with the stablizer of $\lambda$ in the Weyl group of $G$.
Let $\frakp_\la$ be the Lie algebra of $P_\la$. Let $\calP_\lambda$
the moduli scheme of parabolic subgroups of $G$ conjugate to
$P_\lambda$. There is an ample invertible sheaf $\calO(\lambda)$ on
$\calP_\la$, such that $\Gamma(\calP_\la,\calO(\la))^*$ is
isomorphic the Weyl module $W^\la$ of $G$ of highest weight
$\lambda$. Then $\Gamma(\calP_{\la},\calO(\la))$ is
isomorphic to the Schur module $S^{-w_0(\la)}$ of $G$, where $w_0$ is the longest
element in the Weyl group $W$ of $G$.

Next consider the partial Grothendieck alteration
\[\tilde{\frakg}_{\la}=G\times^{P_{\la}}\frakp_\la.\]
It embeds into $\calP_\la\times\frakg$ via
\begin{equation}\label{Groth alteration}
G\times^{P_{\la}}\frakp_\la\to G\times^{P_\la}\frakg\cong
G/P_\la\times\frakg.\end{equation}

Let us use the following notation. If $f:X\to\calP_\la$ is a
morphism, and $\calF$ is a coherent sheaf on $X$, then $\calF\otimes
f^*\calO(\la)$ is denoted by $\calF(\la)$. Let
$p:\calP_\la\times\frakg\to\frakg$ be the projection to the second
factor. We thus obtain a map of $G$-equivariant coherent sheaves
over $\frakg$,
\begin{equation}\label{sur}
p_*\calO_{\calP_\la\times\frakg}(\la)\to p_*\calO_{\tilde{\frakg}_\la}(\la).
\end{equation}
According to \eqref{fun}, there is a functor from the category of $G$-equivariant sheaves on $\frakg$ to the category of $J$-module. Let us denote the $J$-module corresponding to $p_*\calO_{\calP_\la\times\frakg}(\la)$ by $\calS^{-w_0(\la)}$ and to $p_*\calO_{\tilde{\frakg}_\la}(\la)$ by $\calL^{-w_0(\la)}_*$.
Let $\calW^\la$ (resp. $\calL^\la_!$) be the dual
of $\calS^{-w_0(\la)}$ (resp. $\calL_*^{-w_0(\la)}$) as $J$-modules.
We thus obtain a sequence of $J$-module maps
\begin{equation}\label{seqc}
\calL_!^\la\to\calW^\la\to\calS^\la\to\calL_*^\la,
\end{equation}
where $\calW^\la\to\calS^\la$ essentially comes from $W^\la\to
S^\la$.

The main result of this note is
\begin{thm}\label{Main1}
Over $\bbZ_S$, the sequence of maps \eqref{seq} is canonically
identified with \eqref{seqc} as $J$-modules.
\end{thm}

This theorem has the following specialization. Let $e$ be a regular nilpotent element of $\frakg$. Then $(\calP_\la)_e:=e\times_{\frakg}\tilde{\frakg}_\la$ is called the (thick) Springer fiber of $e$, which is isomorphic to the scheme of zero locus of the vector field on $\calP_\la$ determined by $e$.
\begin{cor}Under the isomorphism of $G$-modules,
\[\Gamma(\calP_\la,\calO(\la))\cong S^{-w_0(\la)}\cong H(\Gr,I_*^{-w_0(\la)}),\]
the map $\Gamma(\calP_\la,\calO(\la))\to\Gamma((\calP_\la)_e,\calO(\la))$ corresponds to $H(\Gr,I_*^{-w_0(\la)})\to H(\Gr^{-w_0(\la)})$.
\end{cor}
This corollary was originally conjectured by Ginzburg (cf. \cite[Remark 10]{Bez}).

\subsection{} The main ingredient to prove Theorem \ref{Main1} is
\begin{thm} \label{surjectivity,family version}
Over $\bbZ_S$, the natural map \eqref{sur} is surjective.
\end{thm}

The proof of Theorem \ref{surjectivity,family version}, in turn, relies on the following
result. Let $p$ be a very good prime of $G$ (i.e., those do not
belong to $S$). Then $G\otimes\bar{\bbF}_p$ is the Chevalley group
over $\bar{\bbF}_p$, and we have the corresponding partial
Grothendieck alteration for $G\otimes\bar{\bbF}_p$, which is just
the base change of the partial Grothendieck alteration \eqref{Groth
alteration} of $G$ to $\bar{\bbF}_p$.
We have

\begin{thm}\label{Main2}Let $p>0$ be a very good prime of $G$. Then
over $\bar{\bbF}_p$, the natural embedding
$\tilde{\frakg}_\la\to\calP_\la\times\frakg$ admits a compatibly
Frobenius splitting.
\end{thm}
The basic facts about the Frobenius splitting will be recalled in \S
\ref{generality}.

\subsection{Plan of the paper}In \S \ref{Proof I}, after reviewing the basic facts of the Frobenius splitting, we will prove Theorem \ref{Main2} and Theorem \ref{surjectivity,family version}.  In \S \ref{proof II}, after reviewing the regular centralizer group scheme and the equivariant homology of the affine Grassmannian, we will prove Theorem \ref{Main1}.

\subsection{Further conventions and notations}
As mentioned above, $(G^\vee,B^\vee,T^\vee)$ will denote a complex simple group of adjoint type together with a Borel subgroup and a maximal torus contained in this Borel subgroups. Let
$(G,B,T)$ denote the Langlands dual of $(G^\vee,B^\vee,T^\vee)$ over $\bbZ_S$ and $\frakg\supset\frakb\supset\frakt$ be their Lie algebras. For a weight $\nu$ of $T$ and a base field $k$,
the 1-dimensional representation of $T_k$ (and therefore of $B_k$) corresponding to $\nu$ will
be denoted as $k^{\nu}$. The full flag variety of $G$ is denoted by $\calB=G/B$. The full Grothendieck alteration is denoted by $\tilde{\frakg}=G\times^B\frakb$.

All $G^\vee(\calO)$-equivariant complexes of sheaves on $\Gr$ are taken $\bbZ_S$-coefficients. The (co)homology $H^*(-)$ and $H_*(-)$ are also taken $\bbZ_S$.

For an affine scheme $A$, we will
use $\calO_A$ to denote the ring of functions on $A$. More
generally, if $\calF$ is a quasi-coherent sheaf on $A$, we will
denote the space of its global sections also by $\calF$.

If $H$ is an affine group scheme over some base, and $X$ is an affine $H$-scheme over the base, then we will denote $X/\!\!/H=\Spec\ \calO_X^H$ to be the GIT quotient.

\subsection{Acknowledgement} The author would like to thank
Roman Bezrukavnikov, Edward Frenkel, Dennis Gaitsgory, Joel
Kamnitzer, Shrawan Kumar and Zhiwei Yun for useful discussions.

\section{Proof of Theorem \ref{Main2}, Theorem \ref{surjectivity,family version}}\label{Proof I}
In this section, we prove Theorem \ref{Main2} and its Corollary
\ref{surjectivity,family version}.

\subsection{Generalities on the Frobenius splitting}\label{generality}
In this subsection, $k$ will denote an algebraically closed field of
characteristic $p>0$. Let us briefly recall the general setting.

Let $X$ be a scheme defined over $k$. The Frobenius twist $X'$ of
$X$ is the base change of $X$ along the absolutely Frobenius
morphism of $\Spec k$. Then there is a relative Frobenius morphism
$Fr:X\to X'$.

Recall that (cf. \cite{MR}) $X$ is called Frobenius split if the
$\calO_{X'}$-module map $\calO_{X'}\to Fr_*\calO_X$ admits a
splitting map (i.e. an $\calO_{X'}$-module map
$\varphi:Fr_*\calO_X\to\calO_{X'}$ satisfying $\varphi(1)=1$). Such
a splitting map is called a Frobenius splitting. If $i:Y\to X$ is a
closed embedding defined by an ideal sheaf $\calI_Y$, then $Y$ is
called compatibly split in $X$ if there is a splitting
$\varphi:Fr_*\calO_X\to\calO_{X'}$ which maps $Fr_*\calI_Y$ to
$\calI_{Y'}$. In this case, it induces a split
$\bar{\varphi}:Fr_*\calO_Y\to\calO_{Y'}$.

Assume that $X$ is a smooth scheme over $k$. Then the Grothendieck
duality implies that $\calH om(Fr_*\calO_X,\calO_{X'})\cong
Fr_*\omega_X^{1-p}$, where $\omega_X$ is the canonical sheaf of
$X$. Therefore, we will call a section of $Fr_*\omega_X^{1-p}$ a
\emph{splitting section} if it gives rise to a splitting of $X$
via the above isomorphism. Furthermore, using the Cartier
operator, the above isomorphism was written down explicitly in
\cite{MR}. We recall it in a form we need here.

\begin{lem}\label{criterion of splitting}Let $X=\Spec k[x_1,\ldots,x_n]$.
We will identify the natural map $\calO_{X'}\to Fr_*\calO_X$ as
$k[x_1^p,\ldots,x_n^p]\to k[x_1,\ldots,x_n]$ and
$Fr_*\omega_X^{1-p}$ as $k[x_1,\ldots,x_n](dx_1\wedge\cdots\wedge
dx_n)^{1-p}$. We will adopt the multiple-index notation, so that
$x^a=x_1^{a_1}x_2^{a_2}\cdots x_n^{a_n}$ for
$a=(a_1,\ldots,a_n)\in\bbZ^n$ and $dx=dx_1\wedge\cdots\wedge
dx_n$. Furthermore, if $m$ is an integer, we denote
$\underline{m}=(m,\ldots,m)\in\bbZ^n$. Then the isomorphism
$Fr_*\omega_X^{1-p}\cong\calH om(Fr_*\calO_X,\calO_{X'})$
is given by
\[x^a(dx)^{1-p}(x^b)=\left\{\begin{array}{ll} x^{a+b+\underline{1}-\underline{p}}& p|(a+b+\underline{1}-\underline{p})\\
                                                      0 & \text{otherwise.}\end{array}\right.\]
The same formula holds for $X=\Spec k[[x_1,\ldots,x_n]]$.
\end{lem}

We need another lemma in the sequel. The proof is easy and is left
to the readers.

\begin{lem}\label{compatible with completions} Let $X$ be a scheme
of finite type over $k$ and $Y\subset X$ be a closed subscheme of
$X$. Let $x\in Y$ be a closed point. Since $Fr:X\to X'$ is a
homeomorphism, we can also regard $x$ as a point in $X'$. Let
$\hat{\calO}_{X,x}$ (resp. $\hat{\calO}_{X',x}$, resp.
$\hat{\calI}_{Y,x}$, resp. $\hat{\calI}_{Y',x}$) be the completion
of $\calO_X$ (resp. $\calO_{X'}$, resp. $\calI_Y$, resp.
$\calI_{Y'}$) at $x$. For any $\calO_{X'}$-module map
$f:Fr_*\calO_X\to\calO_{X'}$, let  $f_x:\hat{\calO}_{X,x}\to\hat{\calO}_{X',x}$ be the induced $\hat{\calO}_{X',x}$-module
map. Then: (i) $f$ is a
splitting map if and only if $f_x$ splits the natural map
$\hat{\calO}_{X',x}\to\hat{\calO}_{X,x}$.; and (ii)
$f(Fr_*\calI_Y)\subset\calI_{Y'}$ if and only if
$f_x(\hat{\calI}_{Y,x})\subset \hat{\calI}_{Y',x}$.
\end{lem}

\subsection{Proof of Theorem \ref{Main2}}
In this subsection, we assume that $G$ is simple and simply-connected over an algebraically closed field of characteristic $p>0$. Since all the schemes in the subsection are in fact defined over $\bbF_p$, we will not distinguish them from their Frobenius twists.

Let $P\subset G$ be a
parabolic subgroup and $\frakp\subset\frakg$ be its parabolic
subalgebra, and $\calP=G/P$ be the variety of parabolic subgroups of
$G$ that are conjugate to $P$. Let
$\tilde{\frakg}_\calP:=G\times^P\frakp$ be the partial Grothendieck
alteration. It embeds into $G\times^P\frakg\cong\calP\times\frakg$
as a closed subscheme. If the characteristic $p$ is very good for
$G$, then $\calP$ can be also regarded as the variety of parabolic
subalgebras of $\frakg$ that are conjugate to $\frakp$ (since in
this case the normalizer of $\frakp$ in $G$ is $P$), and
$\tilde{\frakg}_\calP$ can be regarded as the variety of pairs
$(\frakp',\xi)$, where $\frakp'$ is a parabolic subalgebra of
$\frakg$ conjugate to $\frakp$ and $\xi\in\frakp'$. We will prove
the following theorem.

\begin{thm}\label{Frob. splitting}Assume that the characteristic of $k$ is $p>0$ and $p$
is very good for $G$. Then the closed embedding
$\tilde{\frakg}_\calP\to\calP\times\frakg$ admits a compatibly
Frobenius splitting.
\end{thm}
\begin{proof} The proof is divided into several steps.

\medskip

(i) Consider the natural projection
$\pr:\calB\times\frakg\to\calP\times\frakg$. It is a proper
morphism satisfying
$\pr_*\calO_{\calB\times\frakg}=\calO_{\calP\times\frakg}$, which
maps $\tilde{\frakg}_\calB$ onto $\tilde{\frakg}_\calP$.
Therefore, according to \cite[Proposition 4]{MR}, it is enough to prove
the theorem for $\calP=\calB$.

\medskip

(ii) Let $\frakg^*$ be the dual of $\frakg$ so that
$\calO_\frakg=S\frakg^*$ is the symmetric algebra over $\frakg^*$.
It decomposes as $G$-modules according to the natural grading
$S\frakg^*=\sum\limits_{n} S^n\frakg^*$.

Let $\pi:\calB\times\frakg\to\calB$ be the projection. Then
$\omega_{\calB\times\frakg}\cong\pi^*\omega_{\calB}\cong\pi^*\calO_{\calB}(-2\rho)$.
Using the isomorphism $\calB\times\frakg\cong G\times^B\frakg$,
one can identify
\[\begin{split}&\Gamma(\calB\times\frakg,Fr_*\omega_{\calB\times\frakg}^{1-p})\cong\Gamma(\calB,\pi_*\pi^*\omega^{1-p}_{\calB})\\
&\cong(\calO_G\otimes(\calO_\frakg\otimes
k^{-2(p-1)\rho}))^B=\bigoplus_n(\calO_G\otimes(S^n\frakg^*\otimes
k^{-2(p-1)\rho}))^B,\end{split}\] where $\calO_\frakg\otimes
k^{-2(p-1)\rho}$ is regarded as a $B$-module. Let
$d=(p-1)\dim\frakg$. We will use the homogeneous piece
$(\calO_G\otimes(S^d\frakg^*\otimes k^{-2(p-1)\rho}))^B$.

Let us define a natural nonzero $G$-module homomorphism
\[\varepsilon: S^d\frakg^*\to k\]
as follows. Let $I$ be the ideal of $\calO_\frakg$ generated by
$\{v^p, \ v\in\frakg^*\}$. This is a $G$-submodule of
$\calO_\frakg$. Then one has the short exact sequence of
$G$-modules
\begin{equation}\label{ep}0\to S^{d}\frakg^*\cap I\to S^{d}\frakg^*\stackrel{\varepsilon}{\to} k\to 0.\end{equation}
Such a
$G$-module homomorphism gives a $B$-module homomorphism, still
denoted by $\varepsilon$,
\[\varepsilon: S^d\frakg^*\otimes k^{-2(p-1)\rho}\to k^{-2(p-1)\rho}.\]
Therefore, we obtain the following map
\[\begin{split}\ind(\varepsilon):&(\calO_G\otimes(S^d\frakg^*\otimes k^{-2(p-1)\rho}))^B\\
                                 &\to(\calO_G\otimes k^{-2(p-1)\rho})^B=\Gamma(\calB,Fr_*\omega_\calB^{1-p}).\end{split}\]

\begin{lem} A section
$\sigma\in(\calO_G\otimes(S^{d}\frakg^*\otimes
k^{-2(p-1)\rho}))^B\subset\Gamma(\calB\times\frakg,Fr_*\omega_{\calB\times\frakg}^{1-p}))$
is a splitting section of $\calB\times\frakg$ if and only if
$\ind(\varepsilon)(\sigma)$ is a splitting section of $\calB$.
\end{lem}
\begin{proof} The method used here is similar to \cite{KLT}. Let $U_-$ be the unipotent radical of $B_-$, which
is the Borel subgroup of $G$ opposite to $B$. Let $U_-\cdot[1]$ be
the big cell of $\calB$, where $[1]$ denotes our chosen Borel
subgroup $B\subset G$. We choose a system of homogeneous
coordinates $\{x_\alpha,\alpha\in\Delta_+\}$ of $U_-$, i.e.
$U_-=\Spec k[x_\alpha,\alpha\in\Delta_+]$, where $x_\alpha$ is a
$T$-weight function of $U_-$ of weight $\alpha$. Let us also
choose a system of homogeneous coordinates $\{y_i\in\frakg^*,
1\leq i\leq\dim\frakg\}$ for $\frakg$.

An element $\sigma\in(\calO_G\otimes (S^d\frakg^*\otimes
k^{-2(p-1)\rho}))^B\subset\Gamma(\calB\times\frakg,Fr_*\omega_{\calB\times\frakg}^{1-p}))$
restricts over $U_-\cdot[1]$ to an element of the form
\[\text{res}(\sigma)=f(dx\wedge dy)^{1-p}\in Fr_*\omega_{U_-\cdot[1]\times\frakg}^{1-p}\cong\calO_{U_-}\otimes S\frakg^*\otimes k^{-2(p-1)\rho}, \quad f\in\calO_{U_-}\otimes S^d\frakg^*.\]
According to Lemma \ref{criterion of splitting}, this is a
splitting section of $U_-\cdot[1]\times\frakg$ if and only if the
coefficient of the monomial
$x^{\underline{p-1}}y^{\underline{p-1}}$ appearing in $f$ is not
zero and the coefficients of the monomials
$x^{\underline{p-1}+pa}y^{\underline{p-1}+pb} ((a,b)\neq 0\in
\bbZ^{\dim U_-}_{\geq 0}\times\bbZ^{\dim\frakg}_{\geq 0})$
appearing in $f$ are zero. Since $\sigma\in\calO_{U_-}\otimes S^d\frakg^*$, no monomials of the form
$x^ay^{\underline{p-1}+pb} (b\neq 0\in \bbZ^{\dim\frakg}_{\geq
0})$ appear in $f$ (the degree of $y^{\underline{p-1}+pb}(b\neq 0\in \bbZ^{\dim\frakg}_{\geq
0})$ is greater than $d$).

On the other hand, the section
$\ind(\varepsilon)(\sigma)\in(\calO_G\otimes
k^{-2(p-1)\rho})^B=\Gamma(\calB,Fr_*\omega_{\calB}^{1-p})$
restricts over $U_-\cdot[1]$ to an element of the form
\[\text{res}(\ind(\varepsilon)(\sigma))=g(dx)^{1-p}\in Fr_*\omega_{U_-\cdot[1]}^{1-p}\cong\calO_{U_-}\otimes k^{-2(p-1)\rho}, \quad g\in\calO_{U_-}.\]
Again by Lemma \ref{criterion of splitting}, this is a splitting
section of $U_-\cdot[1]$ if and only if the coefficient of the
monomial $x^{\underline{p-1}}$ appearing in $g$ is not zero and
the coefficients of the monomials $x^{\underline{p-1}+pa}, (a\neq
0\in\bbZ^{\dim U_-}_{\geq 0})$ appearing in $g$ are zero.

Now let $\sigma\in (\calO_G\otimes(S^{d}\frakg^*\otimes
k^{-2(p-1)\rho}))^B$. Recall the notations
\[\text{res}(\sigma)=f(dx\wedge dy)^{1-p}, f\in\calO_{U_-}\otimes S^d\frakg^*, \quad\quad \text{res}(\ind(\varepsilon)(\sigma))=g(dx)^{1-p}, g\in\calO_{U_-}.\]
By the following commutative diagram
\[\begin{CD}
(\calO_G\otimes(S^{d}\frakg^*\otimes
k^{-2(p-1)\rho}))^B@>\ind(\varepsilon)>>(\calO_G\otimes
k^{-2(p-1)\rho})^B\\
@V\text{res}VV@V\text{res}VV\\
\calO_{U_-}\otimes(S^{d}\frakg^*\otimes
k^{-2(p-1)\rho})@>\text{id}\otimes\varepsilon>>\calO_{U_-}\otimes
k^{-2(p-1)\rho}
\end{CD}\]
and the definition of $\varepsilon$ (see \eqref{ep}), if we write
$f=f_1y^{\underline{p-1}}+(\text{other} \ \text{terms})$ for some
$f_1\in\calO_{U_-}$, then $g=f_1$.

Therefore, if $\ind(\varepsilon)(\sigma)$ is a splitting section
of $\calB$, then the monomial
$x^{\underline{p-1}+pa}(a\in\bbZ^{\dim U_-}_{\geq 0})$ appears in
$g=f_1$ if and only if $a=0$, which implies the monomial
$x^{\underline{p-1}+pa}y^{\underline{p-1}+pb}((a,b)\in\bbZ^{\dim
U_-}_{\geq 0}\times\bbZ^{\dim\frakg}_{\geq 0})$ appears in $f$ if
and only if $(a,b)=0$. This in turn implies that $\sigma$ is a
splitting section of $\calB\times\frakg$. By the same argument,
the converse holds and the lemma is proven.
\end{proof}

\medskip

(iii) Assume that $\sigma\in (\calO_G\otimes(S\frakg^*\otimes
k^{-2(p-1)\rho}))^B\cong\Hom(Fr_*\calO_{\calB\times\frakg},\calO_{\calB\times\frakg})$.
Let us see when $\sigma(Fr_*\calI_{\tilde{\frakg}})\subset
\calI_{\tilde{\frakg}}$, so that it induces a map
$\sigma:Fr_*\calO_{\tilde{\frakg}}\to\calO_{\tilde{\frakg}}$,
where $\calI_{\tilde{\frakg}}$ is the sheaf of ideals defining
$\tilde{\frakg}\subset\calB\times\frakg$. Let
$\calI_\frakb\subset\calO_\frakg$ denote the ideal defining
$\frakb\subset\frakg$. We define a $B$-submodule $J\subset
\calO_\frakg$ as follows. Observe there is a unique up to scalar
$\calO_\frakg$-module isomorphism $\calO_\frakg\cong\omega_\frakg$.
By Composing it with the isomorphism in Lemma \ref{criterion of
splitting}, we obtain an isomorphism (up to scalar)
$Fr_*\calO_\frakg\cong\Hom(Fr_*\calO_\frakg,\calO_{\frakg})$. We
define
\begin{equation}\label{J}J=\{\delta:Fr_*\calO_\frakg\to\calO_{\frakg}, \quad \delta(Fr_*\calI_\frakb)\subset\calI_{\frakb}\}.\end{equation}
Here is a more concrete description of $J\subset S\frakg^*=Fr_*\calO_\frakg$, from
which the $B$-module structure of $J$ is clear. Fixing the maximal
torus $T$, we have the decomposition $\frakg=\fraku_-+\frakb$. But
as $B$-modules, we only have
\[0\to\fraku^*_-\to\frakg^*\to\frakb^*\to 0.\] Let
$J_1=\text{Im}(S^{(p-1)\dim\fraku_-}\fraku^*_-\otimes S\frakg^*\to
S\frakg^*)$. This is a $B$-submodule of $S\frakg^*$. On the other
hand, let $J_2$ be the ideal of $S\frakg^*$ generated by $\{v^p,\
v\in\fraku^*_-\}$. which is also a $B$-submodule of $S^d\frakg^*$.
Then $J=J_1+J_2$ by Lemma \ref{criterion of splitting}. Observe that $J=\oplus_n J^n$, where $J^n=J\cap
S^n\frakg^*$.
\begin{lem}A section $\sigma\in (\calO_G\otimes(S\frakg^*\otimes
k^{-2(p-1)\rho}))^B\cong\Hom(Fr_*\calO_{\calB\times\frakg},\calO_{\calB\times\frakg})$
maps $Fr_*\calI_{\tilde{\frakg}}$ to $\calI_{\tilde{\frakg}}$ if
and only if $\sigma\in (\calO_G\otimes(J\otimes
k^{-2(p-1)\rho}))^B\subset(\calO_G\otimes(S\frakg^*\otimes
k^{-2(p-1)\rho}))^B$.
\end{lem}
\begin{proof}It is enough to see when
$\sigma(Fr_*\calI_{\tilde{\frakg}})\subset\calI_{\tilde{\frakg}}$
over $U_-\cdot[1]$. Then the lemma follows the first description of
$J$.
\end{proof}

\medskip

(iv) From previous two lemmas, to finish the proof of the theorem,
it is enough to construct a section $\sigma\in
(\calO_G\otimes(J^d\otimes k^{-2(p-1)\rho}))^B$ such that
$\ind(\varepsilon)(\sigma)$ gives a splitting of $\calB$.

Let $\St:=W^{(p-1)\rho}=S^{(p-1)\rho}$ be the first Steinberg module of $G$,
which is irreducible and selfdual. Here $\rho$ is the sum of fundamental weights of $T$. Let us fix a $G$-invariant
non-degenerate bilinear form $(\cdot,\cdot)$ on $\St$. The theorem
then would follow if we could construct a $B$-module homomorphism
\begin{equation}\label{B-map}\gamma:\St\otimes\St\to J^d\otimes k^{-2(p-1)\rho} \quad\text{s.t.}\quad \varepsilon\circ\gamma\neq 0:\St\otimes\St\to k^{-2(p-1)\rho}.\end{equation}
This is because then we would have the following nonzero $G$-module maps
\[\St\otimes\St\stackrel{\ind(\ga)}{\to}(\calO_G\otimes(J^d\otimes k^{-2(p-1)\rho}))^B\stackrel{\ind(\varepsilon)}{\to}(\calO_G\otimes k^{-2(p-1)\rho})^B,\]
and according to the main theorem of \cite{LT}, any
$\sigma=\ind(\ga)(v\otimes w)$ for $v\otimes w\in\St\otimes\St, \
(v,w)\neq 0$ would satisfy our purpose.

\medskip

(v) It remains to construct a $B$-module homomorphism
(\ref{B-map}). However, let us first define a $B$-module
homomorphism
\[\ga_0:\St\otimes\St\otimes k^{2(p-1)\rho}\to\calO_G\] by the following
formula: let $v_+$ (resp. $v_-$) be a nonzero highest (resp.
lowest) weight vector in $\St$, then
\[\quad\quad \ga_0(v\otimes w\otimes v_+\otimes v_+)(g)=(v,gw)(v_+,gv_+), \quad v\otimes w\in\St\otimes\St, g\in G.\]
Since $\omega_G$ is trivial and $\Gamma(G,\calO^*)=k^*$, there is
a unique (up to scalar) $\calO_G$-module isomorphism
$i:\calO_G\cong\omega_G$. Thus, we obtain an isomorphism
$\calO_G\to\omega_G^{1-p}, f\mapsto fi(1)^{1-p}$, and by Lemma
\ref{criterion of splitting}, we obtain a map, still denoted by
$\ga_0$
\[\ga_0:\St\otimes\St\otimes k^{2(p-1)\rho}\to Fr_*\omega_G^{1-p}\cong\Hom(Fr_*\calO_G,\calO_{G}).\]

The main properties of $\ga_0$ is summarized in the following
lemma. Let $\calI_B$ be the ideal defining $B\subset G$.
\begin{lem}\label{property of ga0} For any $\sigma\in\St\otimes\St\otimes
k^{2(p-1)\rho}$, $\ga_0(\sigma)(Fr_*\calI_B)\subset\calI_{B}$.
Furthermore, $\ga_0(v_-\otimes v_-\otimes v_+\otimes v_+)$ is a
splitting section of $G$.
\end{lem}
\begin{proof}Let $U_-B\subset G$ be the open subset of $G$. It is enough to
prove the lemma over $U_-B$. Let us choose a system of homogenous
coordinates $\{x_\alpha,\alpha\in\Delta_+\}$ (resp.
$\{y_\alpha,-\alpha\in\Delta_+$\}) for $U_-$ (resp. for $U$). And
let $t_i$ be the $i$th fundamental weight of $T$.

By construction, $\ga_0(v\otimes w\otimes v_+\otimes
v_+)=(v,gw)(v_+,gv_+)i(1)^{1-p}$. Since $i(1)$ is the unique (up
to scalar) nonzero invariant differential form on $G$,
\[i(1)|_{U_-B}=(\prod_it_i)^{-1}dxdydt\quad \mbox{ up to a scalar}.\] On the
other hand, it is clear that the function $g\mapsto (v_+,gv_+)$
and $g\mapsto (v_-,gv_-)$ restricted to $U_-B$ has the form
\[\begin{split}&f_1(x)(\prod t_i)^{p-1}, f_1(x)\in k[x_\alpha,\alpha\in\Delta_+],\\
& f_2(y)(\prod t_i)^{-(p-1)}, f_2(y)\in
k[y_\alpha,-\alpha\in\Delta_+].\end{split}\] Therefore,
\[\ga_0(v\otimes w\otimes v_+\otimes
v_+)|_{U_-B}=(v,gw)f_1(x)(\prod_i t_i)^{2(p-1)}(dxdydt)^{1-p}, \]
and in particular
\[\ga_0(v_-\otimes v_-\otimes v_+\otimes
v_+)|_{U_-B}=f_2(y)f_1(x)(\prod_i t_i)^{(p-1)}(dxdydt)^{1-p}.\]

Since the $T$-weight of the function $f_1(x)$ is $2(p-1)\rho$, the
monomial
\[x^a=\prod_{\al\in\Delta_+}x_\al^{a_\al},\quad a\in\bbZ_{\geq 0}^{\dim U_-}\]
appearing in $f_1(x)$ will be of the following two forms: either
$\exists\alpha$, s.t. $a_\alpha\geq p$, or $a_\alpha=p-1$ for all
$\alpha\in\Delta_+$. Therefore, according to Lemma \ref{criterion of
splitting}, $\ga_0(v\otimes w\otimes v_+\otimes
v_+)(Fr_*\calI_B|_{U_-B})\subset\calI_{B}|_{U_-B}$. The first
statement of the lemma is proved. Next, it is shown in the proof of
\cite[Lemma 5]{KLT} that the monomial $x^{\underline{p-1}}$ appears
in $f_1(x)$ and none monomials of the form
$x^{\underline{p-1}+pa},a\neq 0\in\bbZ_{\geq 0}^{\dim U_-}$ appear
in $f_1(x)$. Similar results hold for $f_2(y)$. Then by Lemma
\ref{criterion of splitting} again, $\ga_0(v_-\otimes v_-\otimes
v_+\otimes v_+)|_{U_-B}$ gives a splitting of $U_-B$. The second
statement of the lemma follows.
\end{proof}

\medskip

(vi) Finally, let us see how $\ga_0$ gives the desired map as in
(\ref{B-map}). Since the characteristic of $p$ is assumed to be
very good, there is a $G$-equivariant map $\varphi:G\to\frakg$
which sends the unit $e\in G$ to the origin $0\in\frakg$, and induces the identity map
$(D\varphi)_e=\id: T_eG\to T_0\frakg$ (cf. \cite[9.3.2-9.3.3]{BR}).
\begin{lem}\label{B to b}
Such a map will necessarily map $B$ to $\frakb$.
\end{lem}
\begin{proof}Let $2\check{\rho}$ be the sum of positive coroots of $T$. Then we have a morphism $2\check{\rho}:\bbG_m\to T$. We assume that $\bbG_m\subset\bbA^1$ naturally.
Then $\frakb\subset\frakg$ has the following characterization. Let $x\in\frakg(R)$ be any $R$-point of $\frakg$, where $R$ is a $k$-algebra. If the map $(\bbG_m)_R\to\frakg_R$ defined by $t\mapsto \Ad(2\check{\rho}(t))x$ extends to $\bbA^1$, then $x\in\frakb(R)$. Now, the conjugation of $2\check{\rho}(\bbG_m)$ on $B$ extends to a morphism $\bbA^1\times B\to B$. Therefore, by the $G$-equivariance of $\varphi$, $\varphi$ will map $B$ to $\frakb$.
\end{proof}
Therefore, we obtain the following commutative diagram
of $B$-modules
\[\begin{CD}
0@>>>\hat{\calI}_{B,e}@>>>\hat{\calO}_{G,e}@>>>\hat{\calO}_{B,e}@>>>0\\
@.@V\cong VV@V\cong VV@V\cong VV@.\\
0@>>>\hat{\calI}_{\frakb,0}@>>>\hat{\calO}_{\frakg,0}@>>>\hat{\calO}_{\frakb,0}@>>>0.
\end{CD}\]

Now by Lemma \ref{compatible with completions}, we obtain a map of
$B$-modules
\[\begin{split}\St\otimes\St\otimes k^{2(p-1)\rho}\to&\Hom(Fr_*\calO_G,\calO_{G'})\to\Hom(\hat{\calO}_{G,e},\hat{\calO}_{G',e})\\
\cong&\Hom(\hat{\calO}_{\frakg,0},\hat{\calO}_{\frakg',0})\cong Fr_*\hat{\omega}^{1-p}_{\frakg,0}.\end{split}\]
where $\hat{\omega}_{\frakg,0}$ denotes the completion of
$\omega_{\frakg}$ at $0$. The unique
up to scalar isomorphism $\omega_\frakg\to\calO_\frakg$ induces a
$G$-module isomorphism up to scalar
$Fr_*\hat{\omega}^{1-p}_{\frakg,0}\to Fr_*\hat{\calO}_{\frakg,0}\cong\prod_nS^n\frakg^*$.
Therefore, we obtain a map $\ga:\St\otimes\St\otimes
k^{2(p-1)\rho}\to S^d\frakg^*$ by composition of above maps
followed by the projection $\prod_nS^n\frakg^*\to S^d\frakg^*$. By
Lemma \ref{compatible with
completions}, the first part of Lemma \ref{property of ga0}, and the
definition of $J$ (cf. \eqref{J}), we know the image of $\ga$ indeed
lies in $J^d$. By Lemma \ref{compatible with
completions}, and the second part of Lemma \ref{property of ga0}, $\ga(v_-\otimes v_-\otimes
v_+\otimes v_+)$ is a splitting section of $\frakg$, and therefore by Lemma \ref{criterion of splitting}, $(\varepsilon\circ\ga)(v_-\otimes v_-\otimes
v_+\otimes v_+)\neq 0$ (see \eqref{ep} for the definition of
$\varepsilon$).
\end{proof}

\begin{rmk}The existence of Frobenius splitting of $\tilde{\frakg}$ is already
proved (cf. \cite{MvdK} Theorem 3.8 for type $A$ case and cf.
\cite{KLT} Remark 1 for general case). However, to my knowlegdge, the existence of Frobenius splitting of $\tilde{\frakg}_\calP$ is not known before. Our approach is largely inspired by
\cite{KLT}.
\end{rmk}

\subsection{Proof of Theorem \ref{surjectivity,family version}} Now we will deduce
Theorem \ref{surjectivity,family version} as a consequence of the
compatibly Frobenius splitting. First we still assume that
everything is defined over $\bar{\bbF}_p$, with $p$ being very good.
\begin{prop}\label{surj for p good}
(i) The natural map
\[\Gamma(\calP_\la\times\frakg,\calO_{\calP_\la\times\frakg}(\la))\to\Gamma(\tilde{\frakg}_\la,\calO_{\tilde{\frakg}_\la}(\la))\]
is surjective.

(ii) We have
\[H^1(\calP_\la\times\frakg,\calO_{\calP_\la\times\frakg}(\la))=H^1(\tilde{\frakg}_\la,\calO_{\tilde{\frakg}_\la}(\la))=0.\]
\end{prop}
\begin{proof}
Since the closed embedding
$\tilde{\frakg}_\la\to\calP_\la\times\frakg$ is compatibly Frobenius
splitting, by the standard argument (cf. \cite{MR}), it is enough
to prove that
\begin{equation}\label{mlarge}
\Gamma(\calP_\la\times\frakg,\calO_{\calP_\la\times\frakg}(m\la))\to\Gamma(\tilde{\frakg}_\la,\calO_{\tilde{\frakg}_\la}(m\la))
\end{equation}
is surjective, and
\begin{equation}\label{vanish}
H^1(\calP_\la\times\frakg,\calO_{\calP_\la\times\frakg}(m\la))=H^1(\tilde{\frakg}_\la,\calO_{\tilde{\frakg}_\la}(m\la))=0,
\end{equation}
for $m$ sufficiently large. Since $\calP_\la\times\frakg$ is not
projective, to prove this we need a little extra work.

Let $\pi:\calP_\la\times\frakg\to\calP_\la$ be the projection to the
first factor. Then the map \eqref{mlarge} is identified with
\[\Gamma(\calP_\la,\pi_*(\calO_{\calP_\la\times\frakg}(m\la)))\to\Gamma(\calP_\la,\pi_*(\calO_{\tilde{\frakg}_{\la}}(m\la))).\]
By the adjunction formula, the above map is the same as
\[\Gamma(\calP_\la,\pi_*(\calO_{\calP_\la\times\frakg})(m\la))\to\Gamma(\calP_\la,\pi_*(\calO_{\tilde{\frakg}_{\la}})(m\la)),\]
induced from
$S\frakg^*\otimes\calO_{\calP_\la}\otimes\calO(m\la)\to\pi_*\calO_{\tilde{\frakg}_{\la}}\otimes\calO(m\la)$.
It is clear that as $\calO_{\calP_\la}$-modules, there is an
quasi-isomorphism
$S\frakg^*\otimes\Omega^\bullet_{\calP_\la}\otimes\calO(m\la)\to\pi_*\calO_{\tilde{\frakg}_{\la}}\otimes\calO(m\la)$
(the Koszul resolution) extending the above morphism. Therefore,
there is a spectral sequence converging to
$H^{-p+q}(\calP_\la,\pi_*(\calO_{\tilde{\frakg}_{\la}})(m\la))$ with
\[E_1^{-p,q}=H^q(\calP_\la,\Omega_{\calP_\la}^p(m\la))\otimes S\frakg^*\]
Now since $\calO(\la)$ is ample, for $m$ sufficiently large,
$H^q(\calP_\la,\Omega_{\calP_\la}^p(m\la))=0$ for $q>0$. Therefore,
$E_1^{0,0}\twoheadrightarrow E_\infty^{0,0}$. This proves the
surjectivity of \eqref{mlarge} for $m$ sufficiently large. This
argument also implies \eqref{vanish} at the same time.
\end{proof}

Now we begin to prove Theorem \ref{surjectivity,family version}.
Therefore, in the rest paragraph of this subsection, everything is
defined over $\bbZ_S$. We want to show that
\[\Gamma(\calP_\la\times\frakg,\calO_{\calP_\la\times\frakg}(\la))\to\Gamma(\tilde{\frakg}_\la,\calO_{\tilde{\frakg}_\la}(\la))\]
is surjective. It is enough to show that
\[\Gamma(\calP_\la\times\frakg,\calO_{\calP_\la\times\frakg}(\la))\otimes\bar{\bbF}_p\to\Gamma(\tilde{\frakg}_\la,\calO_{\tilde{\frakg}_\la}(\la))\otimes\bar{\bbF}_p\]
is surjective for any $p\not\in S$ (i.e. $p$ is very good). By part
(i) of Proposition \ref{surj for p good}, it is enough to proof that
the canonical morphisms
\[\Gamma(\calP_\la\times\frakg,\calO_{\calP_\la\times\frakg}(\la))\otimes\bar{\bbF}_p\to \Gamma((\calP_\la\times\frakg)\otimes\bar{\bbF}_p,\calO_{\calP_\la\times\frakg}(\la)) \]
and
\[\Gamma(\tilde{\frakg}_\la,\calO_{\tilde{\frakg}_\la}(\la))\otimes\bar{\bbF}_p\to\Gamma(\tilde{\frakg}_\la\otimes\bar{\bbF}_p,\calO_{\tilde{\frakg}_\la}(\la))\]
are isomorphisms. Let us prove the latter isomorphism, since the
former is similar and even simpler. Again, let
$\pi:\calP_\la\times\frakg\to\calP_\la$ be the projection to the
first factor. Then it is equivalent to prove that
\[\Gamma(\calP_\la,\pi_*\calO_{\tilde{\frakg}_\la}(\la))\otimes\bar{\bbF}_p\to\Gamma(\calP_\la\otimes\bar{\bbF}_p,\pi_*\calO_{\tilde{\frakg}_\la}(\la))\]
is surjective. Observe that $\pi_*\calO_{\tilde{\frakg}_\la}(\la)$
is a direct sum of coherent sheaves on $\calP_\la$, each of which is
flat over $\bbZ_S$. By the standard base change theorem for
cohomology, it is enough to show that
$H^1(\calP_\la\otimes\bar{\bbF}_p,\pi_*\calO_{\tilde{\frakg}_\la}(\la))=0$
for every $p\not\in S$, which is the content of part (ii) of
Proposition \ref{surj for p good}. Therefore, Theorem
\ref{surjectivity,family version} is proved.

\subsection{Flatness of $p_*\calO_{\tilde{\frakg}_\la}$ over $\frakg^{reg}$}\label{flat}
Let $\frakg^{reg}$ denote the open subscheme of $\frakg$ consisting
of regular elements. That is, for any $\bbZ_S$-scheme $X$, $\frakg^{reg}(X)$ is the subset of $\frakg(X)$ such that for any point
$x\in X$, the composition $x\to
X\to\frakg$ is a regular element in $\frakg\otimes \kappa(x)$, where $\kappa(x)$ is the residue field of $x$. It is clear that for any very good prime $p$, $\frakg^{reg}\otimes\bar{\bbF}_p\cong\frakg^{reg}_{\bar{\bbF}_p}$.
\begin{prop}\label{freeness}
The restrictions of $p_*\calO_{\calP_\la\times\frakg}(\la)$ and $p_*\calO_{\tilde{\frakg}_\la}(\la)$ to
$\frakg^{reg}$ are locally free.
\end{prop}
\begin{rmk}However, $p_*\calO_{\tilde{\frakg}_\la}(\la)$ is not flat over $\frakg$.
\end{rmk}
\begin{proof}It is clear that $p_*\calO_{\calP_\la\times\frakg}(\la)$ is free on $\frakg$. Therefore, we concentrate on $p_*\calO_{\tilde{\frakg}_\la}(\la)$. It is enough to prove that $p_*\calO_{\tilde{\frakg}_\la\otimes\bar{\bbF}_p}(\la)$ is locally free over $\frakg^{reg}_{\bar{\bbF}_p}$ of the same rank. Therefore, we base change everything to $\bar{\bbF}_p$ without changing the notation, where $p$ is a very good prime of $G$.

Let $L_\la\subset P_\la\subset G$ be the standard Levi
subgroup of $P_\la$. Then the Weyl group of $L_\la$ is $W_\la$, the
stablizer of $\la$ in $W$. We construct a map
$\tilde{\frakg}_\la\to\frakt/\!\!/W_\la=\Spec (S\frakt^*)^{W_\la}$.
Namely, let $U_\la$ be the unipotent radical of $P_\al$ and
$\frakn_\la$ be the nilpotent radical of $\frakp_\la$. Then
$P_\la/U_\la\cong L_\la$ and
$\frakp_\la/\frakn_\la\cong\frakl_\la=\Lie L_\la$. We thus obtain
the following map
\begin{equation}\label{t/w}
G\times^{P_\la}\frakp_\la\to G\times^{P_\la}(\frakp_\la/\frakn_\la)\cong (G/U_\la)\times^{L_\la}\frakl_\la\to [\frakl_\la/L_\la]\to\frakt/\!\!/W_\la,
\end{equation}
where $[\frakl_\la/L_\la]$ is the stack quotient, which maps
naturally to $\frakl_\la/\!\!/L_\la=\Spec (S\frakl_\la^*)^{L_\la}$. To justify the morphism $[\frakl_\la/L_\la]\to\frakt/\!\!/W_\la$, we need to show that the natural map $\frakt/\!\!/W_\la\to\frakl_\la/\!\!/L_\la$ is an isomorphism.

This fact may be well known to experts. However, since we cannot locate a reference, we include a proof here.  It is well known that $\calO_{L_\la}^{L_\la}\cong \calO_T^{W_\la}$. Since the characteristic is very good, according to \cite[9.3.2-9.3.3]{BR}, there is a $G$-equivariant morphism $\varphi:G\to\frakg$ sending the unit $e\in G$ to the origin $0\in\frakg$ and $(D\varphi)_e=\on{id}$. One can argument as in Lemma \ref{B to b} that it must map $L_\la$ to $\frakl_\la$ in an $L_\la$-equivariant way and $T\to \frakt$ in a $W$-equivariant way. We thus obtain the following commutative diagram
\[\xymatrix{T/\!\!/W_\la\ar[d]\ar^\cong[r]&L_\la/\!\!/L_\la\ar[d]\\
\frakt/\!\!/W_\la\ar[r]&\frakl_\la/\!\!/L_\la
}\]
By Theorem 4.1 of \emph{loc. cit.}, the vertical morphisms in the above diagram are \'{e}tale around the unit $e$. Therefore, $\frakt/\!\!/W_\la\to\frakl_\la/\!\!/L_\la$ is \'{e}tale around the origin $0$. Observe that both $(S\frakt^*)^{W_\la}$ and $(S\frakl_\la^*)^{L_\la}$ are positive graded and the map respects the grading. Therefore, $(S\frakl_\la^*)^{L_\la}\cong(S\frakt^*)^{W_\la}$.

We come back to the proof of the proposition. It is easy to see that the map \eqref{t/w} gives rise to the following
commutative diagram
\begin{equation}\label{alteration}\begin{CD}\tilde{\frakg}_{\lambda}@>>> \frakt/\!\!/W_\lambda\\
                @VpVV@VVV\\
                \frakg@>\chi>>\frakt/\!\!/W,\end{CD}\end{equation}
where $\chi:\frakg\to\frakt/\!\!/W$ is the Chevalley map as in Proposition \ref{Chevalley}. Now, the proposition is the consequence of the following two facts
(i) the projection $\frakt/\!\!/W_\la\to\frakt/\!\!/W$ is faithfully
flat over $\bbZ_S$; and (ii) over $\frakg^{reg}$, the above diagram
is Cartesian.

To see (i), we use the result of \cite{De} that the projections $\frakt/\!\!/W_\la$ and
$\frakt\to\frakt/\!\!/W$ are faithfully flat. Therefore,
$\frakt/\!\!/W_\la\to\frakt/\!\!/W$ is also faithfully flat. To see
(ii), let $\tilde{\frakg}_{\lambda}^{reg}=p^{-1}(\frakg^{reg})$. We want to show that
\[\tilde{\frakg}_{\lambda}^{reg}\to\frakg^{reg}\times_{\frakt/\!\!/W}\frakt/\!\!/W_\la\]
is an isomorphism. Since it is proper and quasi-finite, it is finite. Since both schemes are smooth\footnote{The smoothness of $\frakg^{reg}\times_{\frakt/\!\!/W}\frakt/\!\!/W_\la$ follows from the smoothness of $\chi:\frakg^{reg}\to\frakt/\!\!/W$, see \cite[Chap. II, 3.14]{Sl}.}, it is flat. Finally, the assertion follows from the fact that this map is an isomorphism over $\frakg^{reg,ss}\times_{\frakt/\!\!/W}\frakt/\!\!/W_\la$, where $\frakg^{reg,ss}$ is the open subscheme of $\frakg^{reg}$ consisting of regular semi-simple elements.
\end{proof}

An easy application of the above proof is

\begin{prop}\label{isom}
The surjective homomorphism
\begin{equation}\label{surjj}\Gamma(\calP_\la\times\frakg,\calO_{\calP_\la\times\frakg}(\la))\twoheadrightarrow\Gamma(\tilde{\frakg}_{\la},\calO_{\tilde{\frakg}_\la}(\la))\end{equation}
is an isomorphism if and only if $\lambda$ is zero or a minuscule
weight of $G$ w.r.t $B$.
\end{prop}
\begin{proof} If $\lambda=0$,
everything is clear. So we assume that $\lambda$ is not zero in
the following. If $\lambda$ is minuscule, then the morphism of
$\calO_\frakg$-modules \eqref{surjj}
is an isomorphism over the generic point of $\frakg$ by the
following reason. Since the diagram \eqref{alteration} is Cartensian over $\frakg^{reg}$ (one can see easily that the diagram \eqref{alteration} descends to $\bbZ_S$), over the generic point of $\frakg$, $p_*\calO_{\tilde{\frakg}_{\lambda}}(\lambda))$ has rank $|W/W_\la|$, which is the same as the generic rank of $p_*\calO_{\calP_\la\times\frakg}(\la)$ since $\la$ is minuscule. Therefore, the kernel of \eqref{surjj}
is a torsion module over $\calO_\frakg$. However, since
$p_*\calO_{\calP_\la\times\frakg}(\lambda)$ is a free
$\calO_\frakg$-module, the kernel must be zero.

Conversely, if \eqref{surjj}
is an isomorphism. Then the $\bbZ_S$-rank of $\Gamma(\calP_\la,\calO(\la))$, which is the same as the rank of
$p_*\calO_{\calP_\la\times\frakg}(\la)$ as an $\calO_\frakg$-module, is $|W/W_\la|$.
The can happen only when $\lambda$ is minuscule.
\end{proof}

This proposition gives another characterization of minuscule
weights of $\frakg$, which may be generalized to the Kac-Moody algebras.

\section{Proof of Theorem \ref{Main1}}\label{proof II}
\subsection{Regular centralizer}\label{regular}
In this subsection, we review the regular centralizer group scheme of $G$. The regular centralizer group scheme is well-known when $G$ is an algebraic group over a field. We just write down its counterpart for $G$ being a Chevalley group scheme over $\bbZ_S$.

The centralizer group scheme $I$ over $\frakg$ by definition is the group scheme that fits into the following Cartesian diagram
\[\xymatrix{I\ar[r]\ar[d]&\frakg\ar^{\Delta}[d]\\
G\times\frakg\ar^{\Ad\times 1}[r]&\frakg\times\frakg.
}\]
It is easy to see that $\frakg^{reg}$ as defined in \S \ref{flat} is the open subscheme of $\frakg$ over which the fibers of $I$ have the minimal dimension. It is known that $I|_{\frakg^{reg}}$ is commutative. The following proposition is known when the base is a field of very good characteristic (cf. \cite{De,Sl,Ng}). Easy argument will imply that it also hold when the base is $\bbZ_S$.
\begin{prop}\label{Chevalley}
\begin{enumerate}
\item The natural map $\calO_{\frakg}^{G}\to \calO_{\frakt}^W$ is an isomorphism, and they are both isomorphic to a polynomial algebra over $\bbZ_S$. Denote $\frakt/\!\!/W=\Spec\ \calO_{\frakt}^W$, and let $\varpi:\frakt\to\frakc$, $\chi:\frakg\to\frakc$ be the natural projections. Then both $\varpi, \chi$ are faithfully flat. In addition, the restriction $\chi|_{\frakg^{reg}}$ is smooth.
\item There is a (unique up to isomorphism) smooth commutative group scheme $J$ over $\frakc$, such that
\[
\chi^*J|_{\frakg^{reg}}\cong I|_{\frakg^{reg}}.
\]
In literature, $J$ is usually called the regular centralizer group scheme of $G$.
\end{enumerate}
\end{prop}
It is clear from the definition that $[\frakg^{reg}/G]\cong \bbB J$ as stacks over $\frakc$. We have the following natural functor from the category of $G$-equivariant coherent sheaves on $\frakg$ to the category of $J$-modules
\begin{equation}\label{fun}\on{Coh}([\frakg/G])\to\on{Coh}([\frakg^{reg}/G])\cong\on{Coh}(\bbB J)\cong J\on{-mod}.
\end{equation}
In concrete terms, this means that a $G$-equivariant coherent sheaf on $\frakg^{reg}$ descends to a coherent sheaf on $\frakc$ by faithfully flat descent. In addition, such sheaf carry on a natural $J$-action. Let
\begin{equation}\label{Schur}
\calS^{-w_0(\la)}\to\calL_*^{-w_0(\la)}
\end{equation}
be the $J$-module morphism that is the image of \eqref{sur} under this functor. By Proposition \ref{freeness}, as coherent sheaves on $\frakc$, both $\calS^{-w_0(\la)}$ and $\calL_*^{-w_0(\la)}$ are locally free. Let us denote $J$-module map dual to \eqref{Schur} by
\begin{equation}\label{Weyl}
\calL_!^\la\to\calW^\la.
\end{equation}

The following lemma is a consequence of the fact that there is a $G$-module morphism $W^\la\to S^\la$, where $W^\la=\Gamma(\calP_\la,\calO(\la))^*$ is the Weyl module and $S^\la=\Gamma(\calP_{-w_0(\la)},\calO(-w_0(\la))$ is the Schur module.
\begin{lem}There is a natural $J$-module morphism $\calW^\la\to\calS^\la$.
\end{lem}

\subsection{Review of the equivariant homology of the affine Grassmannian}\label{homology}
Let $\Gr=G^\vee(F)/G^\vee(\calO)$ be the affine Grassmannian of
$G^\vee$. This is a union of projective varieties. In \cite{YZ}, the
(equivariant) homology of $\Gr$ with $\bbZ_S$-coefficients is
expressed as the algebraic functions on certain group schemes
associated to $G$. Let us briefly recall it.

First, let $f$ be the unique $W$-invariant quadratic form on $\frakt$ that takes $2$ on long coroots. Then $f$ gives rise to a $W$-equivariant isomorphism $f:\frakt^*\cong\frakt$ over $\bbZ_S$. Therefore, there is an isomorphism
\begin{equation}\label{H(BG)}
\Spec H^*(\bbB G^\vee)\cong \frakt^*/\!\!/W\cong\frakt/\!\!/W
\end{equation}
It is not hard to show that $\Spec H_*^{G^\vee(\calO)}(\Gr)$ is a commutative group scheme over $\frakt/\!\!/W$.

\begin{prop}(See \cite[Proposition 6.6]{YZ}.)There is a canonical isomorphism
\[\Spec\ H_*^{G^{\vee}(\calO)}(\Gr)\cong J.\]
\end{prop}

Now if $\calF\in D_{G^\vee(\calO)}(\Gr)$, then $H_{G^\vee(\calO)}^*(\Gr,\calF)$ is a comodule over
$H^{G^\vee(\calO)}_*(\Gr)$, and therefore a module over $J$.

Let $T^\vee$ be the maximal torus of $G^\vee$. We should also review the $T^\vee$-equivariant homology of $\Gr$. According to \cite{YZ}, the $T^\vee$-equivariant Chern class of the determinantal line bundle of $\Gr$ gives rise to a map
$e^T:\Spec\ H(\bbB T^\vee)\cong\frakt^*\to\frakg^{reg}$ making the following diagram commute
\begin{equation}\label{dia}\xymatrix{\frakt^*\ar^{e^T}[r]\ar_f[d]&\frakg^{reg}\ar^\chi[d]\\
\frakt\ar^{\varpi}[r]&\frakc}
\end{equation}
\begin{prop}(See \cite[Theorem 6.1]{YZ}.) There is a canonical isomorphism of group schemes
\[\Spec H^{T^\vee}_*(\Gr)\cong (e^T)^*I.\]
\end{prop}

\subsection{Proof of Theorem \ref{Main1}}\label{proof main}
We begin with
\begin{prop}\label{a}The natural map
\begin{equation}\label{s}H_{G^\vee(\calO)}(\Gr,I^\la_*)\to H_{G^\vee(\calO)}(\Gr,i^\la_*)\end{equation}
is a surjective map of free $H(\bbB G^\vee)$-modules. Dually, the natural map
\begin{equation}\label{i}H_{G^\vee(\calO)}(\Gr,i_!^\la)\to H_{G^\vee(\calO)}(\Gr,I^\la_!)\end{equation}
is splitting injective as free $H(\bbB G^\vee)$-modules.
\end{prop}
\begin{proof}  We first show that $H_{G^\vee(\calO)}(\Gr,I^\la_*)$ is finitely generated and flat over $H(\bbB G^\vee)$.  Then since $H_{G^\vee(\calO)}(\Gr,I^\la_*)$ is graded, it must be free.

Since $H(\bbB G^\vee)\to H(\bbB T^\vee)$ is finite and faithfully flat by taking $\bbZ_S$-coefficients (see Proposition \ref{Chevalley}), it is enough to show $H_{G^\vee(\calO)}(\Gr,I^\la_*)\otimes_{H(\bbB G^\vee)}H(\bbB T^\vee)$  is finitely generated and flat over $H(\bbB T^\vee)$. First, the natural map
\[H_{G^\vee(\calO)}(\Gr,I^\la_*)\otimes_{H(\bbB G^\vee)}H(\bbB T^\vee)\to H_{T^\vee}(\Gr,I_*^\la)\]
is an isomorphism and by \cite[Lemma 2.2]{YZ}, there is a canonical isomorphism of $H(\bbB T^\vee)$-modules
\[H_{T^\vee}(\Gr,I_*^\la)\cong H(\Gr,I_*^\la)\otimes H(\bbB T^\vee).\]
Secondly, according to \cite[\S 3]{MV}, $H(\Gr,I_*^\la)$ is a free $\bbZ_S$-module of finite rank. Therefore, $H_{G^\vee(\calO)}(\Gr,I^\la_*)\otimes_{H(\bbB G^\vee)}H(\bbB T^\vee)$  is finitely generated and flat over $H(\bbB T^\vee)$. Similarly, one can show that $H_{G^\vee(\calO)}(\Gr,i^\la_*)$, $H_{G^\vee(\calO)}(\Gr,i_!^\la)$ and $H_{G^\vee(\calO)}(\Gr,I^\la_!)$ are also free over $H(\bbB G^\vee)$.

Let $\bbD$ be the Verdier duality functor. It is known by \cite[Proposition 8.1]{MV} that $\bbD I_*^\la=I_!^\la$. Therefore, \eqref{s} implies \eqref{i}. In addition, the Leray spectral sequence $H(\bbB G^\vee, H(\Gr, I^\la_*))\Rightarrow H_{G^\vee(\calO)}(\Gr,I_*^\la)$ (resp. $H(\bbB G^\vee, H(\Gr, i^\la_*))\Rightarrow H_{G^\vee(\calO)}(\Gr,i_*^\la)$) degenerates since all the cohomology concentrate on even degrees. Therefore, it is enough to show that
\[H(\Gr,I_*^\la)\to H(\Gr,i_*^\la)\]
is surjective.

Observe that the natural map $\bbZ_S[\dim\overline{\Gr}^\la]\to i^\la_*$ of complex of sheaves on $\overline{\Gr}^\la$ factors as
\begin{equation*}\bbZ_S[\dim\overline{\Gr}^\la]\to I^\la_*\to i^\la_*.\end{equation*}
Now the proposition follows from the fact that $H(\overline{\Gr}^\la)\to H(\Gr^\la)$ is surjective.
\end{proof}
\begin{rmk}This remark will not be used in the sequel. All sheaves in the remark are taken $\bbZ$-coefficients. Observe that $H^{-\dim\overline{\Gr}^\la}(\Gr,I_!^\la)\cong\bbZ$ and its  unique (up to sign) basis induces $\bbZ[\dim\overline{\Gr}^\la]\to I_!^\la$. One can show that the natural map $i^\la_!\to\bbZ[\dim\overline{\Gr}^\la]$ followed by this $\bbZ[\dim\overline{\Gr}^\la]\to I_!^\la$ is the natural perverse truncation map (up to sign). In other words, we have
\[i_!^\la\to\bbZ[\dim\overline{\Gr}^\la]\to I^\la_!.\]
The proof of the proposition implies that
$H(\Gr,i^\la_!)\to H(\Gr,I^\la_!)$
is splitting injective (as $\bbZ$-modules). But one can show a stronger result holds. Namely, the map $H(\overline{\Gr}^\la,\bbZ[\dim\overline{\Gr}^\la])\to H(\Gr,I^\la_!)$ is splitting injective.
\end{rmk}

\medskip

Now we begin to prove Theorem \ref{Main1}. We first show that there is a commutative diagram of $J$-modules
\[\xymatrix{H_{G^\vee(\calO)}(\Gr,I_!^\la)\ar[r]\ar_\cong[d]& H_{G^\vee(\calO)}(\Gr, I_*^\la)\ar^\cong[d]\\
\calW^\la\ar[r]&\calS^\la
}.\]
It is enough to show that after the base change $\varpi\circ f:\frakt\to\frakc$, such a diagram exists as $(\varpi\circ f)^*J\cong (e^T)^*I$-modules (see \eqref{dia}). On the one hand, by definition, $(\varpi\circ f)^*\calW^\la\to(\varpi\circ f)^*\calS^\la$ as the $(e^T)^*I$-modules is the same as $\calO_{\frakt^*}\otimes W^\la\to\calO_{\frakt^*}\otimes S^\la$. On the other hand,
$(\varpi\circ f)^*H_{G^\vee(\calO)}(\Gr,I_!^\la)\to(\varpi\circ f)^*H_{G^\vee(\calO)}(\Gr, I_*^\la)$ as $(e^T)^*I\cong\Spec\ H^{T^\vee}_*(\Gr)$-modules, is isomorphic to $H_{T^\vee}(\Gr,I_!^\la)\to H_{T^\vee}(\Gr, I_*^\la)$, which is also isomorphic to $\calO_{\frakt^*}\otimes W^\la\to\calO_{\frakt^*}\otimes S^\la$ by \cite{YZ}.

To finish the proof of the theorem, we will show that there exists a $J$-modules isomorphism $H_{G^{\vee}(\calO)}(\Gr,i^\la_*)\to\calL_*^\la$ making the following diagram commute
\[\xymatrix{H_{G^{\vee}(\calO)}(\Gr,I^\la_*)\ar^{\cong}[d]\ar[r]&H_{G^{\vee}(\calO)}(\Gr,i^\la_*)\ar^\cong[d]\\
\calS^\la\ar[r]&\calL_*^\la.
}\]
The isomorphism $H_{G^\vee(\calO)}(\Gr,i_!^\la)\to\calL_!^\la$ is deduced similarly.

All the modules in the above diagram are (locally) free over $\frakt/\!\!/W$ (see Proposition \ref{freeness} and Proposition \ref{a}). In addition, the horizontal maps are surjective by Proposition \ref{a} and Theorem \ref{surjectivity,family version}. Therefore, to show that there is a map $H_{G^{\vee}(\calO)}(\Gr,i^\la_*)\to\calL_*^\la$ making the above diagram commute, it is enough to show there is such a map at the generic point $\frakt/\!\!/W$. We can even just prove that such a map exists over $\xi$, where $\xi$ is the generic point of $\frakt$, which maps to $\frakt/\!\!/W$ via $\xi\to\frakt\stackrel{\varpi}{\to}\frakt/\!\!/W$.

We know that $J_\xi$ is isomorphic to $I_\xi=Z_{G_\xi}(\xi)$, the
centralizer of $\xi$ in $G_\xi$. Since $\xi\in \frakt$,
$Z_\xi(G_\xi)\cong T_\xi$. By the equivariant localization theorem,
the map $H_{G^{\vee}(\calO)}(\Gr,I^\la_*)_\xi\to
H_{G^{\vee}(\calO)}(\Gr,i^\la_*)_\xi$ as modules over $\Spec
H^{G^{\vee}(\calO)}_*(\Gr)_\xi\cong T_\xi$ is killing all the weight
spaces of $H_{G^{\vee}(\calO)}(\Gr,I^\la_*)_\xi$ whose weights are
not in the $W$-orbit of $\la$. On the other hand, the base change of
\eqref{Groth alteration} to $\xi$ is the same as the closed
embedding of the $T_\xi$-fixed point subscheme $(\calP_\la)_\xi$ of
$\calP_\la$ to $\calP_\la$. Therefore, by the localization theorem
of coherent sheaves, the map $(\calS^\la)_\xi\to(\calL_*^\la)_\xi$
as $T_ \xi$-modules also corresponds to killing all the weight
spaces of $(\calS^\la)_\xi$ whose weights are not in the $W$-orbit
of $\la$. It is clear from the above descriptions that a map
$H_{G^{\vee}(\calO)}(\Gr,i^\la_*)_\xi\to(\calL_*^\la)_\xi$ exists.

\end{document}